\documentclass[13pt]{amsart}
\usepackage{amsmath}
\usepackage{amsfonts}
\usepackage{amssymb,enumerate}
\usepackage{amsthm}
\usepackage{hyperref}
\usepackage{centernot}
\usepackage{stmaryrd}

\theoremstyle{plain}
\newtheorem{theorem}{Theorem}[section]

\newtheorem{definition}[theorem]{Definition}

\begin{document}

\title{On finitary Hindman numbers}
\author{Shahram Mohsenipour}
\author{Saharon Shelah}
\address{School of Mathematics, Institute for Research in Fundamental Sciences (IPM)
        P. O. Box 19395-5746, Tehran, Iran}\email{sh.mohsenipour@gmail.com}

\address{The Hebrew University of Jerusalem, Einstein Institute of Mathematics,
Edmond J. Safra Campus, Givat Ram, Jerusalem 91904, Israel}

\address{Department of Mathematics, Hill Center - Busch Campus, Rutgers, The State
University of New Jersey, 110 Frelinghuysen Road, Piscataway, NJ 08854-8019,
USA}\email{shelah@math.huji.ac.il}

\thanks{The research of the first author was in part supported by a grant from IPM (No. 97030403).  The research of the second author was partially supported by European Research Council grant 338821.  This is paper 1146 in Shelah's list of publications.}

\subjclass [2010] {03D10}
\keywords{Ramsey Theory}
\begin{abstract} Spencer asked whether the Paris-Harrington version of the Folkman-Sanders theorem has primitive recursive upper bounds. We give a positive answer to this question.
\end{abstract}
\maketitle
\bibliographystyle{amsplain}
\section{Introduction}
Inspired by Paris-Harrington's strengthening of the finite Ramsey theorem \cite{ph}, Spencer defined in a similar way the following numbers (which we denote by $\textrm{Sp}(m,c)$), strengthening the Folkman-Sanders theorem \cite{soifer} \!\!\footnote{According to Soifer, this should be called the Arnautov-Folkman-Sanders theorem. See \cite{soifer}, pp. 305.}. Let $\textrm{Sp}(m,c)$ be the least integer $k$ such that whenever $[k]=\{1,\dots,k\}$ is $c$-colored then there is $H=\{a_0,\dots,a_{l-1}\}\subset[k]$ such that $\sum H$ (sums of elements of $H$ with no repetition) is monochromatic and $m\leq\min H\leq l$. As in the case of Paris-Harrington's theorem which is deduced from the infinite Ramsey theorem, the existence of the Spencer numbers $\textrm{Sp}(m,c)$ is also easily deduced from the infinite version of the Folkman-Sanders theorem, namely Hindman's theorem \cite{hindman}. Spencer asked whether $\textrm{Sp}(m,c)$ is primitive recursive\footnote{Spencer asked Shelah the question during the workshop: {\bf Combinatorics: Challenges and Applications}, celebrating Noga Alon's 60th birthday, Tel Aviv University, January 17-21, 2016.}. In this paper we give a positive answer to this question. In fact we define the more general numbers $\textrm{Sp}(m,p,c)$ and show that it is in $\mathcal{E}_{5}$ of the Grzegorczyk hierarchy of primitive recursive functions. This means that the rate of the growth of the Spencer function is much slower than the Paris-Harrington function which grows faster than every primitive recursive function. We refer the reader to Section 2.7. of \cite{grs} for getting information about the growth rate of the functions in class $\mathcal{E}_{5}$ which are called WOW functions there. It contains sufficient information to be convinced why our proof implies that the function $\textrm{Sp}(m,p,c)$ is in class $\mathcal{E}_{5}$. We also refer the reader to \cite{erdosmills} for some Ackermannian bounds in both directions for the Paris-Harrington numbers.
\begin{definition}\label{spencer}
For positive integers $m,p,c$, let $\textrm{Sp}(m,p,c)$ be the least integer $k$ such that whenever $[k]=\{1,\dots,k\}$ is $c$-colored then there is $H=\{a_0,\dots,a_{l-1}\}\subset[k]$ $($with $a_0<\dots<a_{l-1}$$)$ such that
\begin{itemize}
\item[(i)] $\sum H$ is monochromatic,
\item[(ii)] $m\leq a_0$, $p\leq l$ and $a_{p-1}\leq l$.
\end{itemize}
\end{definition}
To prove our theorem we use the bounds given in \cite{taylor} for the numbers $\textrm{U}(n,c)$ for the disjoint unions theorem. We also need to consider the finitary Hindman numbers $\textrm{Hind}(n,c)$ defined below. Let's first fix some notations. Let $A,B$ be finite subsets of $\mathbb{N}$, by $A<B$ we mean $\max A<\min B$. If $T$ is a collection of pairwise disjoint sets, then $NU(T)$ will denote the set of non-empty unions of elements $T$. Also by $T=\{A_0,\dots,A_{l-1}\}_{<}$ we mean that the elements of $T$ are finite non-empty subsets of $\mathbb{N}$ and $A_0<\dots<A_{l-1}$. We also need the following notation. Let $A=\{a_0,\dots,a_n\}$ be a finite subset of $\mathbb{N}$. Let $\textrm{exp}_2(A)$ denote $2^{a_0}+\dots+2^{a_n}$. We will use the simple fact that if $A,B$ are two nonempty disjoint finite subsets of $\mathbb{N}$, then $\textrm{exp}_2(A\cup B)=\textrm{exp}_2(A)+\textrm{exp}_2(B)$. Also we have $A\neq B$ iff $\textrm{exp}_2(A)\neq\textrm{exp}_2(B)$. We denote the collection of nonempty subsets of $S$ by $\mathcal{P}^{+}(S)$.

\begin{definition}
For positive integers $n,c$, let $\textrm{U}\,(n,c)$ be the least integer $k$ with the following property. For any disjoint sets $A_0,\dots,A_{k-1}$, if  $\textrm{NU}\,\{A_0,$ $\dots,A_{k-1}\}$ is $c$-colored, then there are disjoint sets $d_0,\dots,d_{n-1}$ such that
\begin{itemize}
\item[(i)] $d_i\in \textrm{NU}\,\{A_0,\dots,A_{k-1}\}$ for $i=0,\dots,n-1$,
\item[(ii)] $\textrm{NU}\,\{d_0,\dots,d_{n-1}\}$ is monochromatic.
\end{itemize}
\end{definition}

\begin{theorem}[Taylor, \cite{taylor}]
$\textrm{U}(n,c)$ is a tower function.
\end{theorem}

\begin{definition}
For positive integers $n,c$, let $\textrm{Hind}\,(n,c)$ be the least integer $k$ such that whenever $\textrm{NU}\,\{A_0,$ $\dots,A_{k-1}\}_{<}$ is $c$-colored, then there is $\{d_0,\dots,d_{n-1}\}_{<}$ such that
\begin{itemize}
\item[(i)] $d_i\in \textrm{NU}\,\{A_0,\dots,A_{k-1}\}_{<}$ for $i=0,\dots,n-1$,
\item[(ii)] $\textrm{NU}\,\{d_0,\dots,d_{n-1}\}_{<}$ is monochromatic.
\end{itemize}
\end{definition}
It is also known that

\begin{theorem}[\cite{dodos}, Proposition 2.19.]
$\textrm{Hind}\,(n,c)$ lies in $\mathcal{E}_{4}$ of the Grzegorczyk hierarchy.
\end{theorem}
\section{Spencer Numbers}
Let $m,p,c$ be positive integers and let $k_{*}=\textrm{Hind}(p+1,c)$. We inductively define a sequence of positive integers $\langle n_{i}; i<k_{*}+1\rangle$ as follows.
\begin{itemize}
\item[(i)] $n_0$ is the least integer with $m\leq 2^{n_0}$,
\item[(ii)] $m_i=2^{\sum_{j=0}^{i}n_{j}}$,
\item[(iii)] $\alpha_{i}=2^{k_{*}-i-1+\sum_{j=1}^{i}n_j}$,
\item[(iv)] $n_{i+1}=\textrm{U}(m_{i},c^{\alpha_{i}})$.
\end{itemize}
\begin{theorem}
For all positive integers $m,p,c$ we have $\textrm{Sp}(m,p,c)\leq 2^{n_{k_{*}}}$.
\end{theorem}
\begin{proof}
Let $\textbf{c}$ be a $c$-coloring of $\{1,\dots,2^{n_{k_{*}}}\}$. We will find $H=\{a_0,\dots,a_{l-1}\}\subseteq[2^{n_{k_{*}}}]$ satisfying the requirements of Definition \ref{spencer}. For $0\leq i\leq k_{*}-1$ we first define the following intervals of positive integers
\[
S_i=[n_0+\dots+n_i,n_0+\dots+n_{i+1}-1].
\]
So $|S_{i}|=n_{i+1}$ and $S_i<S_{i+1}$. Set $S^{*}=\bigcup_{i=0}^{k_{*}-1}S_i$. Let $\textbf{c}^{*}$ be a $c$-coloring of $\mathcal{P}^{+}(S^{*})$ defined by $\textbf{c}^{*}(A)=\textbf{c}(\textrm{exp}_2(A))$. For the next step, we shall find specific disjoint subsets $w_{i,s}\subseteq S_{i}$ for $0\leq i\leq k_{*}-1$, $0\leq s< m_i$ by reverse induction on $0\leq i\leq k_{*}-1$. Let $\textbf{c}_i$ be a coloring of $\mathcal{P}^{+}(S_i)$ defined as follows. For every $u,v\in\mathcal{P}^{+}(S_i)$, we put $\textbf{c}_i(u)=\textbf{c}_i(v)$ if for all $A\in\mathcal{P}(\bigcup_{j<i}S_j)$ and all $B\subseteq\{i+1,\dots,k_{*}-1\}$, we have
\begin{equation}\label{eq1}
\textbf{c}^{*}\big{(}A\cup u\cup\displaystyle\bigcup_{j\in B}w_{j,0}\big{)}=
\textbf{c}^{*}\big{(}A\cup v\cup\displaystyle\bigcup_{j\in B}w_{j,0}\big{)}.
\end{equation}
As $|\mathcal{P}(\bigcup_{j<i}S_j)|=2^{\sum_{j=1}^{i}n_j}$ and $|\mathcal{P}(\{i+1,\dots,k_{*}-1\})|=2^{k_{*}-i-1}$, we observe that the number of colors of $\textbf{c}_i$ is at most $c^{\alpha_{i}}$ where $\alpha_{i}=2^{k_{*}-i-1+\sum_{j=1}^{i}n_j}$. So from $n_{i+1}=\textrm{U}(m_{i},c^{\alpha_{i}})$ it follows that there are disjoint subsets $w_{i,s}\subseteq S_{i}$ for $0\leq s< m_i$ such that $NU\{w_{i,0},\dots,w_{i,m_{i}-1}\}$ is $\textbf{c}_i$-monochromatic. It is clear by construction that for $i_1<i_2$ we have $w_{i_1,j_1}<w_{i_2,j_2}$. Now consider
\[
NU\{w_{0,0},w_{1,0},\dots,w_{k_{*}-1,0}\}_{<}
\]
with the coloring $\textbf{c}^{*}$. Recall that $k_{*}=\textrm{Hind}(p+1,c)$, then there is $\{v_0,\dots,v_p\}_{<}$ such that
\begin{itemize}
\item[(i)] $v_{i}\in NU\{w_{0,0},w_{1,0},\dots,w_{k_{*}-1,0}\}_{<}$ for $0\leq i\leq p$,
\item[(ii)] $NU\{v_{0},\dots,v_{p}\}_{<}$ is $\textbf{c}^{*}$-monochromatic.
\end{itemize}
Assume that $v_{p}=w_{e_1,0}\cup\dots\cup w_{e_r,0}$ and $l^{*}=m_{e_1}$. Now set
\[
v_{p+1}=w_{e_1,1}\cup\dots\cup w_{e_r,1},
\]
\[
v_{p+2}=w_{e_1,2}\cup\dots\cup w_{e_r,2},
\]
\[
.\,\,\,\,\,\,\,.\,\,\,\,\,\,.\,\,\,\,\,\,.\,\,\,\,\,\,\,.\,\,\,\,\,\,.
\]
\[
\,\,\,\,\,\,\,\,v_{p+l^{*}-1}=w_{e_1,l^{*}-1}\cup\dots\cup w_{e_r,l^{*}-1}.
\]
Note that $v_{0},\dots,v_{p+l^{*}-1}$ are disjoint. We claim the desired $H=\{a_{0},\dots,a_{l-1}\}$ is obtained by putting $l=p+l^{*}$ and $a_{i}=\textrm{exp}_{2}(v_i)$. First observe that
\[
a_{0}=\textrm{exp}_{2}(v_0)\geq 2^{n_0}\geq m.
\]
Let $v_{p-1}=w_{d_1,0}\cup\dots\cup w_{d_q,0}$. Also $v_{p-1}<v_{p}$ implies $d_q<e_1$, so we have
\begin{eqnarray*}
              a_{p-1}=\textrm{exp}_{2}(v_{p-1})&=&\textrm{exp}_{2}(w_{d_1,0})+\dots+\textrm{exp}_{2}(w_{d_q,0})\\
                                             &\leq&\textrm{exp}_{2}(S_{d_1})+\dots+\textrm{exp}_{2}(S_{d_q})\\
                                               &\leq&2^{n_0}+2^{n_0+1}+\dots+2^{n_0+n_1+\dots+n_{d_q+1}-1}\\
                                               &\leq&2^{n_0+n_1+\dots+n_{d_q+1}}\,\,=\,\,m_{d_q+1}\,\,\leq\,\, m_{e_1}\,\,=\,\, l^{*}\,\,\leq\,\, l.
\end{eqnarray*}
Note that $a_0<a_1<\dots<a_{p-1}$, and also $a_{p-1}<a_{i}$ for $i\geq p$. This is enough for our purpose and there is no need to know the order of $\{a_p,a_{p+1},\dots,a_{l-1}\}$. It remains to show that $\sum H$ is $\textbf{c}$-monochromatic. This is equivalent to saying that $NU\{v_{0},\dots,v_{l-1}\}$ is $\textbf{c}^{*}$-monochromatic. Recall that $NU\{v_{0},\dots,v_{p}\}$ is $\textbf{c}^{*}$-monochromatic.  Let
\[
A_1\in NU\{v_{0},\dots,v_{p-1}\},\,\,\, B_1\in\{A_1,\emptyset\},\,\,\, A_2\in NU\{v_{p},\dots,v_{l-1}\}.
\]
Obviously $\textbf{c}^{*}(A_1)=\textbf{c}^{*}(v_p)$. So we will finish if we show $\textbf{c}^{*}(B_1\cup A_2)=\textbf{c}^{*}(v_p)$. This will be done by iterated application of the relation (\ref{eq1}) when $u,v\in NU\{w_{i,0},\dots,$ $w_{i,m_{i}-1}\}$. First note that we can write $A_2$ as
\[
\displaystyle\bigcup_{i\in I}w_{e_1,i}\cup\displaystyle\bigcup_{i\in I}w_{e_2,i}\cup\ldots\displaystyle\cup\bigcup_{i\in I}w_{e_r,i}
\]
for some $I\subseteq\{0,1,\dots,l^{*}-1\}$. Finally
\begin{eqnarray*}
 \textbf{c}^{*}(v_p)= \textbf{c}^{*}(B_1\cup v_p)&=&\textbf{c}^{*}\big{(}B_1\cup w_{e_1,0}\cup w_{e_2,0}\cup\dots\cup w_{e_r,0}\big{)}\\
                                             &=&\textbf{c}^{*}\big{(}B_1\cup\displaystyle\bigcup_{i\in I}w_{e_1,i}\cup w_{e_2,0}\cup\dots\cup w_{e_r,0}\big{)}\\
                                               &=&\textbf{c}^{*}\big{(}B_1\cup\displaystyle\bigcup_{i\in I}w_{e_1,i}\cup\displaystyle\bigcup_{i\in I}w_{e_2,i}\cup\dots\cup w_{e_r,0}\big{)}=\cdots\\
                                               &=&\textbf{c}^{*}\big{(}B_1\cup\displaystyle\bigcup_{i\in I}w_{e_1,i}\cup\displaystyle\bigcup_{i\in I}w_{e_2,i}\cup\ldots\displaystyle\cup\bigcup_{i\in I}w_{e_r,i}\big{)}\\
                                               &=&\textbf{c}^{*}(B_1\cup A_2).
 \end{eqnarray*}
\end{proof}
\bibliography{reference}
\bibliographystyle{plain}
\end{document}